\documentclass[12pt,a4paper]{article}
\usepackage{amsmath,amssymb,amsthm}
\usepackage{amssymb,latexsym, bbm,comment}
\usepackage{graphicx}
\usepackage{makeidx}
\newtheorem{theorem}{Theorem}[section]
\newtheorem{prop}[theorem]{Proposition}%[section]
\newtheorem{cor}[theorem]{Corollary}%[section]
\numberwithin{equation}{section}

\newcommand{\R}{\mathbb{R}}
\newcommand{\Rd}{\mathbb{R}^d}
\newcommand{\nN}{n \in \mathbb{N}}
\newcommand{\N}{\mathbb{N}}
\newcommand{\C}{\mathbb{C}}

\newcommand{\g}{\mathfrak{g}}

\newcommand{\p}{\mathfrak{p}}
\newcommand{\fk}{\mathfrak{k}}
\newcommand{\tr}{\mbox{tr}}

\newcommand{\bean}{\begin{eqnarray*}}
\newcommand{\eean}{\end{eqnarray*}}

\newcommand{\la}{\langle}
\newcommand{\ra}{\rangle}

\newcommand{\Z}{\mathbb{Z}}

\newcommand{\G}{\widehat{G}}

\begin{document}

\date{}

\title{Transition Densities and Traces for Invariant Feller Processes on Compact Symmetric Spaces}

\author{David Applebaum, Trang Le Ngan\\ School of Mathematics and Statistics,\\ University of
Sheffield,\\ Hicks Building, Hounsfield Road,\\ Sheffield,
England, S3 7RH\\ ~~~~~~~\\e-mail: D.Applebaum@sheffield.ac.uk, tlengan1@sheffield.ac.uk}

\maketitle

\begin{abstract}
We find necessary and sufficient conditions for a finite $K$--bi--invariant measure on a compact Gelfand pair $(G, K)$ to have a square--integrable density. For convolution semigroups, this is equivalent to having a continuous density in positive time. When $(G,K)$ is a compact Riemannian symmetric pair, we study the induced transition density for $G$--invariant Feller processes on the symmetric space $X = G/K$. These are obtained as projections of $K$--bi--invariant L\'{e}vy processes on $G$, whose laws form a convolution semigroup.  We obtain a Fourier series expansion for the density, in terms of spherical functions, where the spectrum is described by Gangolli's L\'{e}vy--Khintchine formula. The density of returns to any given point on $X$ is given by the trace of the transition semigroup, and for subordinated Brownian motion, we can calculate the short time asymptotics of this quantity using recent work of Ba\~nuelos and Baudoin. In the case of the sphere, there is an interesting connection with the Funk--Hecke theorem.
\end{abstract}

\section{Introduction}

Let $X = (X(t), t \geq 0)$ be a (time--homogeneous) Feller--Markov process, which takes values in a locally compact space $M$ that is equipped with a positive regular Borel measure $\mu$ on its Borel $\sigma$--algebra. Key quantities of interest are the transition kernel $K_{t}(x, A)$ which is the probability that $X(t) \in A$ given that $X(0) =x$, where $x \in M$ and $A$ is a Borel set, and the Feller semigroup $(P_{t}, t \geq 0)$ defined on the space of real--valued continuous functions on $M$ that vanish at infinity by
$$ P_{t}f(x) = \int_{M}f(y)K_{t}(x, dy).$$
There are a number of interesting and fundamental (linked) questions that we can ask:
\begin{enumerate}
\item[(I)] Does a transition density $k_{t}(\cdot,\cdot)$ exist for $t > 0$, so that $$K_{t}(x, A) = \int_{A}k_{t}(x,y)\mu(dy),$$ and does the function $k: (0, \infty) \times M \times M \rightarrow [0, \infty)$ have good regularity properties, such as continuity, differentiability or finite $L^{p}$-norm?
\item[(II)] Do the operators $P_{t}$ extend to form a semigroup on $L^{2}(M, \mu)$. If so, when is it self--adjoint or trace--class?
\item[(III)]  Does there exists a complete set of (normalised) eigenfunctions $\{\phi_{n}, \nN\}$ for $P_{t}$ (with $t > 0$) so that for all $x,y \in M$ we can write
\begin{equation} \label{pre}
k_{t}(x, y) = \sum_{\nN}\lambda_{n}(t)\phi_{n}(x)\phi_{n}(y),
\end{equation}
where $P_{t}\phi_{n} = \lambda_{n}(t)\phi_{n}$?
  \item[(IV)] When do we have a trace formula:  $$\int_{M}k_{t}(x,x)\mu(dx) = \mbox{trace}(P_{t})?$$
\end{enumerate}

Assuming (I) and some imposed regularity, a quite general approach was taken to (II) and (III) in \cite{Get} with $M$ assumed to be a compact separable metric space. A key role here is played by the requirement that $k_{t}$ exists and is square--integrable. The case where $X$ is a compact Riemannian manifold, $\mu$ is the Riemannian volume measure $\mu$ and $X$ is Brownian motion has been extensively studied. In this case $k_{t}$ is the ``ubiquitous'' heat kernel \cite{JL}, and (I) to (IV) all have positive answers (see e.g. Chapter 3 of \cite{Ros} or Chapter VI of \cite{Chav}).

In \cite{App3} (and references therein) positive answers to (1) to (IV) were obtained for a class of central symmetric L\'{e}vy processes in compact Lie groups (where $\mu$ is normalised Haar measure), which are obtained by subordinating Brownian motion. Here the key techniques used were harmonic analytic, arising from the representation theory of $G$. The restriction to central (conjugate--invariant) processes was key as it ensured that the Fourier transform of the measure was a scalar, and these scalars give us the required eigenvalues, through the L\'{e}vy--Khintchine formula.

In this paper, we present another class of processes for which (I) to (IV) are valid. The object of interest is a Feller--Markov process defined on a compact symmetric space $M$, which has $G$--invariant transition probabilities, where $G$ is the identity component of the isometry group of $M$, and is conditioned to start at the point fixed by a closed subgroup $K$ of $G$ (so $M = G/K)$.  The reference measure $\mu$ is the unique (normalised) $G$--invariant measure on $M$. It is well--known that all such processes arise as the projection to $M$ of a L\'{e}vy process in $G$ whose laws form a $K$--bi--invariant convolution semigroup of probability measures (see \cite{Berg1,Liao, LiaoN}). We cannot assert that this is a special case of the theory developed in \cite{App3}, as the measures we consider are not, in general, central (see Proposition \ref{nocent1}); however many of the techniques developed in \cite{App3}, may be applied here. In particular, we find that the eigenvalues we need are given by Gangolli's L\'{e}vy--Khintchine formula \cite{Gang1, LW}, and the eigenvectors are the spherical functions (so we deal with a complex form of (\ref{pre})).

The plan of the paper is as follows. In section 2, we develop some general considerations concerning measures on homogeneous spaces. Some, but not all, of the results presented there are known. The purpose of section 3, is to extend the work of \cite{App2} to find necessary and sufficient conditions, in terms of the Fourier transform, for a finite measure associated to a compact Gelfand pair to have a square--integrable density. In section 4, we consider convolution semigroups of probability measures, where we present a recent result of Liao \cite{LiaoN} which tells us that the measures have a continuous density if and only if they have a square--integrable one. So from the work of sections 3 and 4 together, we have necessary and sufficient conditions for a measure within a convolution semigroup, as above, to have a continuous density (in positive time). In section 5 we specialise to compact symmetric spaces, where we develop the Fourier expansion of the transition density, and obtain the required trace formula. It is worth pointing out that we don't require our measures to be symmetric (or equivalently $P_{t}$ to be self--adjoint), in contrast to \cite{Get} and \cite{App3}. We also apply the theory of \cite{BaBa} to study short--time asymptotics of the transition density corresponding to subordinated Brownian motion on $M$. Finally in section 6, we present the example of the $n$--sphere in a little more detail, and consider the implications of the Funk--Hecke theorem within our context.

\vspace{5pt}

{\bf Notation.} If $X$ is a locally compact Hausdorff space, then ${\mathcal B}(X)$ is its Borel $\sigma$--algebra,  $C_{c}(X)$ will denote the linear space of continuous functions from $X$ to $\R$ having compact support, ${\mathcal M}(X)$ is the linear space of positive regular Borel measures on $(X, {\mathcal B}(X))$, and ${\mathcal M}_{F}(X)$ is the subspace comprising finite measures. If $G$ is a locally compact Hausdorff group, we will denote its neutral element by $e$. We equip the space ${\mathcal M}(G)$ with the binary operation of convolution $*$, so that if $\mu_{1}, \mu_{2} \in {\mathcal M}(G)$, their convolution $\mu_{1} * \mu_{2}$ is the unique element of ${\mathcal M}(G)$ such that for all $f \in C_{c}(G)$,
$$ \int_{G}f(g)(\mu_{1} * \mu_{2})(dg) = \int_{G}\int_{G}f(gh)\mu_{1}(dg)\mu_{2}(dh).$$
Then ${\mathcal M}(G)$ is a monoid, with neutral element given by the measure $\delta_{e}$, where for all $A \in {\mathcal B}(G), \delta_{e}(A): = \left\{ \begin{array}{c c} 1 & \mbox{if}~e \in A\\
0 & \mbox{if}~e \notin A \end{array} \right.$. If $\mu \in {\mathcal M}(G)$, then $\mu^{\prime} \in {\mathcal M}(G)$, where $\mu^{\prime}(A): = \mu(A^{-1})$ for all $A \in {\mathcal B}(G)$, and $A^{-1}:= \{g^{-1}; g \in G\}$.

\noindent All $L^{p}$ spaces appearing in this paper comprise complex--valued functions.

\noindent The space of all $d \times d$ complex--valued matrices is denoted by $M_{d}(\C)$, and the trace of $A \in M_{d}(\C)$ is written tr$(A)$.

\section{Absolute Continuity of Measures on Homogeneous Spaces}

Let $G$ be a locally compact Hausdorff group, which we equip with a a left Haar measure $m_{G}$. We will tend to write $m_{G}(dg) = dg$ within integrals. The modular homomorphism from $G$ to the multiplicative group $(0, \infty)$ will be denoted $\Delta_{G}$. It is uniquely defined by the fact that
\begin{equation} \label{mod}
\int_{G}f(gh^{-1})dg = \Delta_{G}(h)\int_{G}f(g)dg,
\end{equation}
for all $h \in G, f \in C_{c}(G)$.

Let $K$ be a closed subgroup of $G$, with fixed Haar measure $m_{K}$, and $X$ denote the homogeneous space $G/K$ of left cosets of $G$, i.e. $X = \{gK, g \in G\}$. We equip $X$ with the usual (Hausdorff) topology which is such that the canonical surjection $\xi:G \rightarrow X$ is both continuous and open. We will write $o: = \xi(K)$. The group $G$ acts on $X$ by homeomorphisms  via the action
$$ \tau(g)g^{\prime}K = gg^{\prime}K,$$ for all $g, g^{\prime} \in G$. We will make frequent use of the fact that for all $g \in G$,
\begin{equation} \label{use}
 \xi \circ l_{g} = \tau(g) \circ \xi,
\end{equation}

\noindent where $l_{g}(h) = gh$ for all $g,h \in G$. If $K$ is compact, we will always normalise so that $m_{K}(K) = 1$.

Define $P:C_{c}(G) \rightarrow C_{c}(X)$ by
$$ (Pf)(gK): = \int_{K}f(gk)dk,$$
for each $g \in G, f \in C_{c}(G)$. It is shown that $P$ is surjective in \cite{Foll}, pp.61--2.

Let $C_{c,K}(G)$ denote the linear subspace of $C_{c}(G)$ comprising functions that are $K$--right--invariant. If $K$ is compact, then any $F \in C_{c}(X)$ gives rise to $F^{\xi} \in C_{c,K}(G)$ by the assignment
$F^{\xi} = F \circ \xi$. It is not difficult to see that this induces a linear isomorphism between the two spaces. If $K$ is not compact, then $\xi^{-1}(C_{c}(X))$% \subseteq C_{K}(G)$
may contain no functions of compact support (e.g. consider the case $G = \R$ and $K = \Z$).

If $\mu \in {\mathcal M}(G)$ then $\mu_{\xi}:= \mu \circ \xi^{-1} \in {\mathcal M}(X)$. Moreover for all $F \in C_{c}(X)$ we have (see e.g. \cite{Bog}, Proposition 3.6.1, pp. 190--1)
\begin{equation} \label{int1}
\int_{G}F^{\xi}(g)\mu(dg) = \int_{X}F(x)\mu_{\xi}(dx),
\end{equation}
provided $F^{\xi} \in L^{1}(G, \mu)$.

If we begin with functions defined on $G$, rather than on $X$, then given $\mu \in {\mathcal M}(G)$, there exists a unique $\widetilde{\mu} \in {\mathcal M}(X)$ such that for all $f \in C_{c}(G)$,
\begin{equation} \label{int2}
\int_{G}f(g)\mu(dg) = \int_{X}(Pf)(gK)\tilde{\mu}(dgK),
\end{equation}
if and only if for all $f \in C_{c}(G), k \in K$,

\begin{equation} \label{int3}
\int_{G}f(gk^{-1})\mu(dg) = \Delta_{K}(k)\int_{G}f(g)\mu(dg),
\end{equation}
(see \cite{Reit}, pp. 157--8 or \cite{Weil} pp. 42--5).

\noindent Some simple consequences of (\ref{int3}) are:
\begin{enumerate}
\item If $\mu$ is a non--trivial $K$--right--invariant measure on $G$, then $\widetilde{\mu}$ exists if and only if $\Delta_{K}(k) = 1$ for all $k \in K$.
\item If (\ref{int3}) holds, then taking $f$ therein to be $K$--right--invariant, we see that either $\Delta_{K}(k) = 1$ for all $k \in K$, or $\int_{G}f(g)\mu(dg) = 0$ for all $f \in C_{c,K}(G)$.
\end{enumerate}

\noindent If we take $\mu = m_{G}$, then it follows from (\ref{int3}) and (\ref{mod}) that $\widetilde{m_{G}}$ exists and is unique if and only if
\begin{equation} \label{mod2}
\Delta_{K}(k) = \Delta_{G}(k),
\end{equation}
for all $k \in K$. In this case we will write $\sigma: = \widetilde{m_{G}}$. It is easily seen from (\ref{int2}) that $\sigma$ is $G$-invariant, in that
$$ \sigma(\tau(g)A) = \sigma(A),$$
for all $g \in G, A \in {\mathcal B}(X)$, and in fact, if it exists, $\sigma$ is (up to multiplication by a non--negative constant) the unique such measure on $(X, {\mathcal B}(X))$. If $K$ is compact, then (\ref{mod2}) holds with both sides of the equation being equal to one.

We could repeat the above discussion, with $X$ being replaced by $X^{\prime} = K \backslash G$, the space of right cosets of $G$, which is again a locally compact Hausdorff space, with topology such that the natural surjection
$\xi^{\prime}: G \rightarrow X^{\prime}$ is open and continuous. The natural action of $G$ on $X^{\prime}$ is $\tau^{\prime}(h)(Kg) = Kgh$, for $g,h \in G$, and if (\ref{mod2}) holds, there is a unique (up to non--negative scalar multiplication) $G$--invariant measure $\sigma^{\prime}$ on $X^{\prime}$ such that $\sigma^{\prime}(\tau(g^{\prime})A) = \sigma^{\prime}(A),$
for all $g \in G, A \in {\mathcal B}(X)$. This measure is related to right Haar measure on $G$ in the same way that $\sigma$ is related to left Haar measure. The mapping $C$ which takes $gX$ to $Xg$ for all $g \in G$ is easily seen to be a homeomorphism between $X$ and $X^{\prime}$, and we then have $\sigma^{\prime} = \sigma  \circ C^{-1}$.

From now on we will always assume that $K$ is compact. For $f \in C_{c}(G), g \in G$, let $P_{K}f(g): = \int_{K}f(gk)dk \in C_{c,K}(G)$. Then for all $F \in C_{c}(G), P_{K}F = PF \circ \xi$.  Let ${\cal M}_{K}(G)$ denote the subset of ${\mathcal M}(G)$ comprising measures that are $K$--right--invariant. Any measure $\mu \in {\mathcal M}_{K}(G)$, is determined (through the Riesz representation theorem) by its action on the space $C_{c,K}(G)$ since for all $f \in C_{c}(G)$,
\begin{equation} \label{int4}
\int_{G}f(g)\mu(dg) = \int_{G}P_{K}f(g)\mu(dg).
\end{equation}
If (\ref{int3}) is satisfied, and $\mu \in {\mathcal M}_{K}(G)$, then it follows from (\ref{int1}), (\ref{int2}) and (\ref{int4})  that  $\widetilde{\mu} = \mu_{\xi}$. In particular, $m_{G} \in {\mathcal M}_{K}(G)$, as can be seen from (\ref{mod}) and the fact that $\Delta_{G}(k) = 1$ for all $k \in K$ (see also Proposition 1.10 (a) in \cite{LiaoN}), so $\sigma = (m_{G})_{\xi}$.

 The mapping $\mu \rightarrow \mu_{\xi}$ is an isometric isomorphism between ${\mathcal M}_{K}(G)$ and ${\mathcal M}(X)$, and $f \rightarrow f \circ \xi$ is an isometric isomorphism between $L^{p}_{K}(G):=L^{p}_{K}(G, m_{G})$ with $L^{p}(X):=L^{p}(X, \sigma)$ for $p \geq 1$.
For each $g,h \in G$ define $L_{g}f(h) = f(g^{-1}h)$ for $f \in L^{p}(G)$. It is well--known (and easy to deduce) that $L_{g}$ is an isometric isomorphism of $L^{p}(G)$, and if $1  \leq p < \infty$, then the mapping $g \rightarrow L_{g}f$ is continuous from $G$ to $L^{p}(G)$ (see e.g. Proposition 1.2.1 in \cite{App1}). For each $g \in G, F \in L^{p}(X)$, define $T_{g}F: = F \circ \tau(g^{-1})$. Then $T_{g}$ is an isometric isomorphism. Moreover we have

\begin{prop} \label{cont}
If $ 1 \leq p < \infty$, for each $F \in L^{p}(X)$ the mapping $g \rightarrow T_{g}F$ is continuous from $G$ to $L^{p}(X)$.
\end{prop}

\begin{proof} Since for all $g, h \in G$, $T_{gh} = T_{g} \circ T_{h}$ and $T_{h}$ is an isometry, its sufficient to prove continuity at $e$. First observe that by (\ref{use})
\bean \int_{X}|T_{g}F(x) - F(x)|^{p}\sigma(dx) & = & \int_{G}|(T_{g}F)^{\xi}(g^{\prime}) - F^{\xi}(g^{\prime})|^{p}dg^{\prime} \\
& = & \int_{G}|(F \circ \tau(g^{-1}) \circ \xi) (g^{\prime}) - F^{\xi}(g^{\prime})|^{p}dg^{\prime} \\
& = & \int_{G}|L_{g}F^{\xi}(g^{\prime}) - F^{\xi}(g^{\prime})|^{p}dg^{\prime}, \eean
and the result follows by continuity of the map $g \rightarrow L_{g}F^{\xi}$. \end{proof}

Now suppose that $\mu \in {\mathcal M}_{K}(G)$ is absolutely continuous with respect to $m_{G}$ and write the Radon--Nikodym derivative $h: = d\mu/dm_{G}$. In the sequel we will frequently identify $h$ with a particular member of the equivalence class that it defines in $L^{1}(G)$, and in common with standard probabilistic usage, we may refer to any version of $h$ as the {\it density} of $\mu$ (with respect to $m_{G}$).

\begin{prop} \label{RNri}
If $\mu$ is $K$--right--invariant, then $h$ is $K$--right--invariant almost everywhere.
\end{prop}

\begin{proof} For all $f \in C_{c}(G), k \in K$, using (\ref{mod}), we have
\bean \int_{G}f(g)h(gk)dg & = & \int_{G}f(gk^{-1})h(g)\Delta_{G}(k^{-1})dg\\
& = & \int_{G}f(gk^{-1})\mu(dg) \\
& = & \int_{G}f(g)\mu(dg) = \int_{G}f(g)h(g)dg, \eean
and the result follows. \end{proof}

\begin{prop} \label{RN1} The measure $\mu \in {\cal M}_{K}(G)$ is absolutely continuous with respect to Haar measure on $G$, having Radon Nikodym $H^{\xi}$ if and only if $\mu_{\xi}$ is absolutely continuous with respect to the $G$--invariant measure $\sigma$ on $X$, having Radon-Nikodym derivative $H$. Furthermore $H$ is continuous/$L^{p}$ for $ 1 \leq p < \infty$, if and only if $H^{\xi}$ is.
\end{prop}

\begin{proof} Firstly let $\mu \in {\cal M}_{K}(G)$ be absolutely continuous as stated. Then its Radon--Nikodym derivative $h$ is $K$--right-invariant by Proposition \ref{RNri}, and so $h = H^{\xi}$ for some $H \in L^{1}(X)$. Then for all $F \in C_{c}(X)$,
\bean \int_{X}F(x)\mu_{\xi}(dx) & = & \int_{G}F^{\xi}(g)H^{\xi}(g)dg\\
& = & \int_{X}F(x)H(x)\sigma(dx), \eean
and the result follows.

Conversely, if $\mu_{\xi}$ is absolutely continuous with Radon--Nikodym derivative $H$, then
\bean \int_{G}F^{\xi}(g)\mu(dg) & = & \int_{X}F(x)H(x)\sigma(dx)\\
& = & \int_{G}F^{\xi}(g)H^{\xi}(g)dg, \eean
and the result again follows since $\mu$ is determined by its action on $C_{c, K}(G)$.
The result on continuity follows from the fact that the mapping $H \rightarrow H^{\xi}$ is a bijection between $C(X)$ and $C_{K}(G)$. The integrability statement follows similarly.
\end{proof}

We have the following partial generalisations of a known result on locally compact groups due to Raikov--Williamson (see \cite{Wehn} and \cite{App1}, Theorem 4.4.1 p.98).

\begin{prop} \label{RW}
If $\mu \in {\cal M}(X)$ is absolutely continuous with respect to $\sigma$ then for all $E \in {\mathcal B}(X),~\mu(\tau(g)E) \rightarrow \mu(E)$ as $g \rightarrow e$.
\end{prop}

\begin{proof} Writing $h : = d\mu/d\sigma$, we have
\bean |\mu(\tau(g)E) - \mu(E)| & \leq & \int_{E}|h(\tau(g^{-1})x) - h(x)|\sigma(dx)\\
& \leq & ||T(g)h - h||_{L^{1}(X)} \rightarrow 0~\mbox{as}~ g \rightarrow e, \eean
by Proposition \ref{cont}. \end{proof}

We conjecture that the converse of Proposition \ref{RW} also holds, but the proof of the corresponding result on a group $G$ requires both the left and right action of $G$ on itself, and we do not have analogues of both tools available to us.

The space $E:=K \backslash G /K$ is the set of all double cosets $\{KgK, g \in G\}$. Note that each such double coset is an orbit of $K$ in $X$, i.e. if $x = gK$ then
 $$ KgK = \{\tau(k)x; k \in K\}.$$
The set $E$ is a locally compact Hausdorff space when equipped with the topology which makes the canonical surjection $\zeta$ from $G$ to $E$ continuous and open. If $N$ is the normaliser of $K$ in $G$, then there is an action $\omega$ of $N$ on $E$ so that $\omega(n)KgK = K\tau(n)gK$, for all $g \in G, n \in N$. It is shown in Corollary 3.2 of \cite{Liu} that there is an invariant measure $\alpha$ on $E$ in that for all $n \in N$,
\begin{equation} \label{doubleinv}
\int_{E}f(\omega(n)x)\alpha(dx) = \int_{E}f(x)\alpha(dx).
\end{equation}
Furthermore, (see Theorem 2.1 in \cite{Liu}), for all $f \in C_{c}(E)$,
$$ \int_{E}f(x)\alpha(dx) = \int_{G}(f \circ \zeta)(g)dg.$$

\section{Square--Integrability of Densities on Compact Gelfand Pairs}

 From now on, we will assume that $(G, K)$ forms a compact Gelfand pair, so that $G$ is a compact group, $K$ is a closed subgroup, and the Banach algebra (with respect to convolution) $L^{1}(K \backslash G/K)$ is commutative (see e.g. \cite{Wol} for background on such structures, and for material that now follows). Haar measure $m_{G}$ will be normalised henceforth, so that $m_{G}(G) = 1$. We denote by $\G$ the {\it unitary dual} of $G$, i.e. the set of all equivalence classes of irreducible representations of $G$, with respect to unitary conjugation. If $\pi \in \G$, its representation space $V_{\pi}$ is finite--dimensional, and we will write $d_{\pi}:=$dim$(V_{\pi})$. The celebrated Peter--Weyl theorem tells us that $\{\sqrt{d_{\pi}}\pi_{i,j}; 1 \leq i, j \leq d_{\pi}, \pi \in \G\}$ is a complete orthonormal basis for $L^{2}(G)$.

Let $L^{2}_{lK}(G), L^{2}_{rK}(G)$ and $L^{2}_{bK}(G): = L^{2}_{lK}(G)\cap L^{2}_{rK}(G)$ be the subspaces of $L^{2}(G)$ comprising functions that are almost--everywhere $K$--left--invariant, $K$--right--invariant, and $K$--bi--invariant (respectively). The orthogonal projections from $L^{2}(G)$ onto these spaces will be denoted, respectively $P^{K}, P_{K}$ and $Q_{K} = P^{K}P_{K} = P_{K}P^{K}$, so that for all $f \in L^{2}(G), g \in G$,
$$ P^{K}f(g) = \int_{K}f(kg)dk, P_{K}f(g) = \int_{K}f(gk)dk, Q_{K}f(g) = \int_{K}\int_{K}f(kgk^{\prime})dkdk^{\prime}.$$
We can and will use natural isomorphisms between these spaces to identify $L^{2}_{lk}(G)$ with $L^{2}(X^{\prime}, \sigma^{\prime}),L^{2}_{rk}(G)$ with $L^{2}(X, \sigma)$ and $L^{2}_{bk}(G)$ with $L^{2}(E, \alpha)$. This last space will often just be written as $L^{2}(K \backslash G /K)$, in line with standard usage.

Now recall that a representation $\pi$ of $G$ is said to be spherical if there exists a non--zero spherical vector $u_{\pi} \in V_{\pi}$, i.e. $\pi(k)u_{\pi} = u_{\pi}$ for all $k \in K$. If this is the case, then $u_{\pi}$ is unique up to scalar multiplication, and we define $V_{\pi}^{K}: = \{\lambda u_{\pi}; \lambda \in \C\}$. We find it convenient to define $V_{\pi}^{K} = \{0\}$ if $\pi$ is not spherical. In either case, let $E_{\pi}^{K} = \int_{K}\pi(k)dk$ be the orthogonal projection from $V_{\pi}$ to $V_{\pi}^{K}$. When $\pi$ is spherical, we will, for convenience, assume that $u_{\pi}$ has norm one, and we choose an orthonormal basis $\{e_{1}^{\pi}, \ldots, e_{d_{\pi}}^{\pi}\}$ in $V_{\pi}$, with $e_{1}^{\pi} = u_{\pi}$. Let $\G_{s}$ be the subset of $\G$ comprising spherical representations.

Now let $f^{\pi}_{u, v}: = \la \pi(\cdot)u, v \ra$, where $u, v \in V_{\pi}, \pi \in \G$. Then $f^{\pi}_{u,v} \in C(G)$, and easy algebra yields
\begin{equation} \label{uses}
P^{K}f^{\pi}_{u,v} =  f^{\pi}_{u, E_{\pi}^{K}v}, P_{K}f^{\pi}_{u,v} =  f^{\pi}_{E_{\pi}^{K}u, v}, Q_{K}f^{\pi}_{u,v} = f^{\pi}_{E_{\pi}^{K}u, E_{\pi}^{K}v}.
\end{equation}

We have the following consequences of the Peter--Weyl theorem:

\begin{prop} \label{PW}
\begin{enumerate}
\item $\{\sqrt{d_{\pi}} \la \pi(\cdot)e_{i}^{\pi}, e_{1}^{\pi} \ra; i = 1, \ldots, d_{\pi}, \pi \in \G_{s}\}$ is a complete orthonormal basis for $L^{2}_{lK}(G)$.
\item $\{\sqrt{d_{\pi}} \la \pi(\cdot)e_{1}^{\pi}, e_{j}^{\pi} \ra; j = 1, \ldots, d_{\pi},\pi \in \G_{s}\}$ is a complete orthonormal basis for $L^{2}_{rK}(G)$.
\item $\{\sqrt{d_{\pi}} \la \pi(\cdot)e_{1}^{\pi}, e_{1}^{\pi} \ra; \pi \in \G_{s}\}$ is a complete orthonormal basis for $L^{2}_{bK}(G)$.
\end{enumerate}
\end{prop}

\begin{proof} This follows easily from the Peter--Weyl theorem and (\ref{uses}). Note that at least (3) is well--known (see e.g. \cite{Wol} Proposition 9.10.4, p.205 and \cite{Hel2}, Theorem 3.5, pp.533--4.)
\end{proof}

In relation to Proposition \ref{PW}(3), observe that the prescription $$\phi_{\pi}(g): =  \la e_{1}^{\pi}, \pi(g) e_{1}^{\pi} \ra,$$ for $g \in G$ defines a (positive--definite) spherical function on $G$, i.e. a non--trivial continuous function from $G$ to $\C$ so that for all $g, h \in G$,

\begin{equation} \label{spher}
\int_{K}\phi_{\pi}(gkh)dk = \phi_{\pi}(g)\phi_{\pi}(h),
\end{equation}

\noindent and all spherical functions on $G$ arise in this way (see \cite{Hel2} pp.414--7 or \cite{Wol} pp.204--5). Since the conjugate representation to $\pi$ is both irreducible and spherical whenever $\pi$ is, we can rewrite the result stated in the more familiar form that
$\{\sqrt{d_{\pi}}\phi_{\pi}, \pi \in \G_{s}\}$ is a complete orthonormal basis for $L^{2}_{bK}(G)$.

If $\mu \in {\mathcal M}_{F}(G)$, its Fourier transform is the matrix--valued function
$$ \widehat{\mu}(\pi) = \int_{G}\pi(g^{-1})\mu(dg),$$ where $\pi \in \widehat{G}$.
Properties of the Fourier transform are developed in section 4.2 of \cite{App1}. In particular, $\widehat{\mu}$ uniquely determines the measure $\mu$.

If $\mu$ is $K$-bi--invariant, its {\it spherical transform} is the complex-valued mapping:
$$ \widehat{\mu}(\phi) = \int_{G}\phi(g)\mu(dg),$$
where $\phi$ is a spherical function on $G$, and this also uniquely determines $\mu$ (see e.g. \cite{Hey1}).

In \cite{App2} (see also Theorem 4.5.1 in \cite{App1}), it is shown that $\mu \in {\mathcal M}_{F}(G)$ has a square--integrable density if and only if

\begin{equation} \label{Apref}
\sum_{\pi \in\ G}d_{\pi}||\widehat{\mu}(\pi)||_{HS}^{2} = \sum_{\pi \in \G}d_{\pi}\sum_{i,j = 1}^{d_{\pi}}|\widehat{\mu}(\pi)_{ij}|^{2} < \infty,
\end{equation}
where $|| \cdot ||_{HS}$ denotes the matrix Hilbert--Schmidt norm, so that $||\widehat{\mu}(\pi)||_{HS}^{2}: = \tr(\widehat{\mu}(\pi)^{*}\widehat{\mu}(\pi))$. Furthermore, if (\ref{Apref}) holds, then the density $f_{\mu}:= d\mu/dg$ has the $L^{2}$--Fourier expansion:
\begin{equation} \label{dens1}
f_{\mu} = \sum_{\pi \in \G}d_{\pi}\tr(\widehat{\mu}(\pi)\pi).
\end{equation}

 We will need the following useful characterisations of $K$--invariant measures by means of their Fourier transforms, in relation to which it's worth noting that the mapping $\mu \rightarrow \mu^{\prime}$ is a bijection between $K$--left--invariant and $K$--right--invariant measures on $G$. % where $\mu^{\prime}(A): = \mu(A^{-1})$ for all $A \in {\mathcal B}(G)$.

\begin{prop} \label{FT1} Let $\mu \in {\mathcal M}_{F}(G)$.
\begin{enumerate}
\item The following are equivalent:
\begin{enumerate}
\item The measure $\mu$ is $K$--left--invariant,
\item $\widehat{\mu}(\pi)E^{K}_{\pi} = \left\{ \begin{array}{c c} \widehat{\mu}(\pi)&~\mbox{for all}~ \pi \in \G_{s}\\
                                                                  0 &~\mbox{for all}~ \pi \notin \G_{s}, \end{array} \right.$
\item $\widehat{\mu}(\pi)_{ij} = 0$ for all $\pi \notin \G_{s}$, or $\pi \in \G_{s}$ and $i \neq 1$.
\end{enumerate}
\item

\begin{enumerate}
\item The measure $\mu$ is $K$--right--invariant,
\item $E^{K}_{\pi}\widehat{\mu}(\pi)  = \left\{ \begin{array}{c c} \widehat{\mu}(\pi)&~\mbox{for all}~ \pi \in \G_{s}\\
                                                                  0 &~\mbox{for all}~ \pi \notin \G_{s}, \end{array} \right.$
\item $\widehat{\mu}(\pi)_{ij} = 0$ for all $\pi \notin \G_{s}$, or $\pi \in \G_{s}$ and $j \neq 1$.
\end{enumerate}
\item
\begin{enumerate}
\item The measure $\mu$ is $K$--bi--invariant,
\item $E^{K}_{\pi}\widehat{\mu}(\pi)E^{K}_{\pi} = \left\{ \begin{array}{c c} \widehat{\mu}(\pi)&~\mbox{for all}~ \pi \in \G_{s}\\
                                                                  0 &~\mbox{for all}~ \pi \notin \G_{s}, \end{array} \right.$
\item $\widehat{\mu}(\pi)_{ij} = 0$ for all $\pi \notin \G_{s}$, or $\pi \in \G_{s}$ and $i,j \neq 1$.
\end{enumerate}

\end{enumerate}
\end{prop}

\begin{proof} We just prove (2) as (1) is similar and (3) follows from these two assertions. The equivalence of (b) and (c) is straightforward linear algebra and is left to the reader. To show that (a) implies (b),
let $\pi \in \G$ and $u, v \in V_{\pi}$. If $\mu$ is $K$--right--invariant, then
\bean \la \widehat{\mu}(\pi)u, v \ra & = & \int_{G}f^{\pi}_{u,v}(g^{-1})\mu(dg) \\
& = &  \int_{G}P^{K}f^{\pi}_{u,v}(g^{-1})\mu(dg)\\
& = & \int_{G}f^{\pi}_{u, E^{K}_{\pi}v}(g^{-1})\mu(dg)\\
& = &  \la \widehat{\mu}(\pi)u, E^{K}_{\pi}v \ra\eean
and the result follows.

 To show that (c) implies (a), if $\widehat{\mu}(\pi)_{ij} = 0$ for all $\pi \notin \G_{s}$ or $\pi \in \G_{s}$ and $j \neq 1$, then $\int_{G} \la \pi(g)e_{j}, e_{i} \ra \mu(dg) = 0$. So by Proposition \ref{PW}(2), $\int_{G}f_{u,v}^{\pi}(g)\mu(dg) = 0$ unless the function $f_{u,v}^{\pi}$ is $K$--right--invariant. Hence, by the Peter--Weyl theorem for continuous functions (see e.g. Theorem 2.2.4 in \cite{App1}, p.33), $\mu$ is determined by its integrals against functions in $C_{rc}(G)$, and so it is $K$--right--invariant. \end{proof}

When we combine the main result of \cite{App2} (see also Theorem 4.5.1 in \cite{App1}) with that of Proposition \ref{FT1} we get

\begin{theorem} \label{L2} Let $\mu \in {\mathcal M}_{F}(G)$.
\begin{enumerate}
\item If $\mu$ is $K$--left--invariant, then it has an $L^{2}$--density if and only if
$$ \sum_{\pi \in \G_{s}}d_{\pi}||\widehat{\mu}(\pi)||^{2}_{HS} = \sum_{\pi \in \G_{s}}d_{\pi}\sum_{j= 1}^{d_{\pi}}|\widehat{\mu}(\pi)_{1j}|^{2}  < \infty.$$

\item If $\mu$ is $K$--right--invariant, then it has an $L^{2}$--density if and only if
$$ \sum_{\pi \in \G_{s}}d_{\pi}||\widehat{\mu}(\pi)||^{2}_{HS} = \sum_{\pi \in \G_{s}}d_{\pi}\sum_{i = 1}^{d_{\pi}}|\widehat{\mu}(\pi)_{i1}|^{2} < \infty.$$

\item If $\mu$ is $K$--bi--invariant, then it has an $L^{2}$--density if and only if
$$ \sum_{\pi \in \G_{s}}d_{\pi}|\widehat{\mu}(\phi_{\pi})|^{2} < \infty.$$

\end{enumerate}
\end{theorem}

\begin{proof} We just prove (1) as the others are similar. If $\mu$ has an $L^{2}$--density, then the result follows from (\ref{Apref}) and Proposition \ref{FT1}(1). For the converse direction, note that by $K$--left--invariance of $\mu$ and Proposition \ref{FT1}(1), we have
$$ \sum_{\pi \in \G}d_{\pi}\sum_{i,j = 1}^{d_{\pi}}|\widehat{\mu}(\pi)_{ij}|^{2} = \sum_{\pi \in \G_{s}}d_{\pi}\sum_{j= 1}^{d_{\pi}}|\widehat{\mu}(\pi)_{1j}|^{2}  < \infty,$$
and then the result again follows by (\ref{Apref}). \end{proof}

In all three cases, the Fourier expansion of the density is given by (\ref{dens1}). Any (left, right or bi)--$K$--invariance of the measure is inherited by the density (almost everywhere). This is a consequence of Proposition \ref{RNri} and its generalisation to the $K$--left--invariant case; it can also be deduced by uniqueness of Fourier transforms, using Proposition \ref{FT1}. %by showing that (e.g. in the $K$--left--invariant case), $P^{K}f_{\mu} = f_{\mu}$.
If $\mu$ is $K$--bi--invariant, it is easy to check that (in the $L^{2}$--sense), for all $\pi \in \G_{s}$,
\begin{equation} \label{dens2}
f_{\mu} = \sum_{\pi \in \G_{s}}d_{\pi}\overline{\widehat{\mu}(\phi_{\pi})}\phi_{\pi}.
\end{equation}

{\bf Remarks.} \begin{enumerate} \item  The advantage of these results over (\ref{Apref}) is the reduction in summing over the whole of $\G$ to summing over the subset $\G_{s}$, and then summing over a smaller number of matrix elements; in (3) there is the additional advantage of having a single matrix element.

\item The results of this section generalise beyond the category of Gelfand pairs, to arbitrary $(G, K)$, where $G$ is compact and $K$ is closed; but the set--up presented here is convenient for the sequel.
\end{enumerate}
%\vspace{5pt}

%{\bf Note.} All of the results of this section hold in the case where $G$ is a general compact group. But in that case, we only interpret $\phi_{\pi}$ as a spherical function if $(G, K)$ is a Gelfand pair.

\section{$K$--Invariant Densities and Kernels for Convolution Semigroups}

Let $(\mu_{t}, t \geq 0)$ be a convolution semigroup of probability measures on the compact group $G$. By this we mean that
\begin{itemize}
\item $\mu_{s+t} = \mu_{s} * \mu_{t}$ for all $s, t \geq 0$,
\item $\mbox{weak-lim}_{t \rightarrow 0}\mu_{t} = \mu_{0}$.
\end{itemize}
It then follows that $\mu_{0}$ is Haar measure on a compact subgroup of $G$ (see Theorem 4.6.1 in \cite{App1}, p.104). We say that the convolution semigroup is {\it standard} if $\mu_{0} = \delta_{e}$. In this case,
(see e.g. Proposition 5.1.2 in \cite{App1} and the discussion that follows) for each $\pi \in \G, (\widehat{\mu_{t}}(\pi), t \geq 0)$ is a strongly continuous one--parameter contraction semigroup on $V_{\pi}$. Furthermore $(P_{t}, t \geq 0)$ is a contraction semigroup of linear operators on $L^{2}(G)$ defined for each $t \geq 0, f \in L^{2}(G), g \in G$ by
$$ P_{t}f(g) = \int_{G}f(gh)\mu_{t}(dh),$$
and we have $\widehat{\mu_{t}}(\pi)_{ij} = P_{t}\overline{\pi_{ji}}(e)$ for each $\pi \in \G, 1 \leq i,j \leq d_{\pi}$.

It is shown in \cite{L1} that a convolution semigroup is $K$--left--invariant if and only if it is $K$--right--invariant if and only if it is $K$--bi--invariant. Then $\mu_{0} = m_{K}$, $(P_{t}, t \geq 0)$ as defined above\footnote{It also acts as contractions on $L^{2}(G)$, and is a semigroup in the sense that $P_{s+t} = P_{s}P_{t}$ for all $s, t \geq 0$, but $P_{0} = P_{K}$ in this case.} is a contraction semigroup on $L^{2}(K \backslash G /K)$, and for each $\pi \in \G_{s}, (\widehat{\mu_{t}}(\phi_{\pi}), t \geq 0)$ is a strongly continuous one--parameter contraction semigroup of complex numbers.

For each $\mu \in {\mathcal M}(G)$, we have $\mu^{(K)}:= \mu \circ Q_{K}^{-1} \in {\mathcal M}(K \backslash G /K)$ and we note that for all $f \in C_{c}(G)$,
\begin{equation} \label{Bib}
\int_{G}f(g)\mu^{(K)}(dg) = \int_{G}(Q_{K}f)(g)\mu(dg).
\end{equation}

\begin{theorem} \label{inher} Let $\mu \in {\mathcal M}_{F}(G)$. \begin{enumerate} \item If $\mu$ has a square--integrable or continuous density then so does $\mu^{(K)}$.
 \item If $G$ is a connected Lie group having Lie algebra $\g$ and $\mu$ has a $C^{p}$--density for $p \in \N$, then so does $\mu^{(K)}$.
\end{enumerate}
In all cases, if $f$ is the density of $\mu$, then that of $\mu^{(K)}$ is $Q_{K}f$.
\end{theorem}

\begin{proof} \begin{enumerate} \item Both follow easily since $Q_{K}$ is an orthogonal projection from $L^{2}(G)$ to $L^{2}(K \backslash G /K)$ which preserves continuity.
\item For all $X_{1}, \ldots, X_{p}\in \g, g \in G$ the mapping $g \rightarrow X_{1} \ldots X_{p}Q_{K}f(g)$ is well defined and continuous, indeed standard arguments yield
$$ X_{1} \ldots X_{p}Q_{K}f(g) = \int_{K}\int_{K} X_{1} \ldots X_{p}f(kgk^{\prime})dkdk^{\prime} = Q_{K}(X_{1} \ldots X_{p}f)(g),$$
and the result follows by a theorem of Sugiura \cite{Sug} pp. 42--3 (see also Theorem 1.3.5 on p.20 of \cite{App1}). \end{enumerate} \end{proof}

Now suppose that $(\mu_{t}, t \geq 0)$ is a standard convolution semigroup, and consider the associated set of $K$--bi--invariant probability measures $(\mu_{t}^{(K)}, t \geq 0$).

\begin{prop} \label{inher1} $(\mu_{t}^{(K)}, t \geq 0)$ is a $K$--bi--invariant convolution semigroup on $G$.
\end{prop}

\begin{proof} For all $f \in C(G), s, t \geq 0$, by (\ref{Bib}),
\bean \int_{G}f(g)\mu_{s +t}^{(K)}(dg) & =  & \int_{G}\int_{G}Q_{K}f(gh)\mu_{s}(dg)\mu_{t}(dh) \\
& = & \int_{G}\int_{G}\int_{K}\int_{K}f(kghk^{\prime})dkdk^{\prime}\mu_{s}(dg)\mu_{t}(dh)\\
& = & \int_{G}\int_{G}\int_{K}\int_{K}\int_{K}f(kgll^{-1}hk^{\prime})dldkdk^{\prime}\mu_{s}(dg)\mu_{t}(dh)\\
& = & \int_{G}\int_{G}\int_{K}\int_{K}\int_{K}\int_{K}f(kgll^{\prime}l^{-1}hk^{\prime})dl^{\prime}dldkdk^{\prime}\mu_{s}(dg)\mu_{t}(dh)\\
& = & \int_{G}\int_{G}\int_{K}\int_{K}\int_{K}\int_{K}f(kgll^{\prime}hk^{\prime})dl^{\prime}dldkdk^{\prime}\mu_{s}(dg)\mu_{t}(dh)\\
& = & \int_{G}f(g)(\mu_{s}^{(K)} * \mu_{t}^{(K)})(dg)
\eean

Here we have used the bi--invariance of Haar measure on $K$ to first make a change of variable $l \rightarrow ll^{\prime}$ and then $l^{\prime} \rightarrow l^{\prime}l$.
The fact that $\mu_{0}^{(K)} = m_{K}$ follows from

$$ \int_{G}f(g)\mu_{0}^{(K)}(dg) =  \int_{G}(Q_{K}f)(g)\mu_{0}(dg) = \int_{K}\int_{K}f(kl)dkdl = \int_{K}f(k)dk. $$
The weak continuity follows easily from (\ref{Bib}).

\end{proof}

Many explicit examples of convolution semigroups that we consider in the next section fall under the aegis of Theorem \ref{inher} and Proposition \ref{inher1}.

Recall that $\mu \in {\mathcal M}(G)$ is said to be {\it central} (or conjugate--invariant) if $\mu(gAg^{-1}) = \mu(A)$ for all $g \in G, A \in {\mathcal B}(G)$, and we let ${\mathcal M}^{F}_{c}(G)$ be the set of all finite central measures on $G$. It is shown in Theorem 4.2.2 of \cite{App1} that $\mu \in {\mathcal M}^{F}_{c}(G)$ if and only if for each $\pi \in \G$ there exists $c_{\pi} \in \C$ so that $\widehat{\mu}(\pi) = c_{\pi}I_{\pi}$. It then follows that $\widehat{\mu^{(K)}}(\phi_{\pi}) = c_{\pi}$ for all $\pi \in G_{s}$. We will see important examples of central measures in the next section. If $G$ is abelian, then all measures on $G$ are central, and all irreducible representations of $G$ are one--dimensional. The next proposition presents some evidence that if $G$ is compact and non--abelian and $\mu \in {\mathcal M}^{F}_{c}$ is non--trivial, then $\mu^{(K)}$ is not central, in general.\footnote{In private e--mail communication with the authors, Ming Liao has produced an example of a non--trivial measure on a compact group that is both central and $K$--bi--invariant.}%Indeed if $c_{\pi} \neq 0$, for some $\pi \in \G_{s}$, then $\mu^{(K)}$ cannot be central by Proposition \ref{FT1} (3). To be more precise

\begin{prop} \label{nocent1} Let $\mu \in {\mathcal M}^{F}_{c}(G)$, so that $\widehat{\mu}(\pi) = c_{\pi}I_{\pi}$ for all $\pi \in \G$, and assume that there exists $\pi \in \G_{s}$ with dim$(V_{\pi}) > 1$. If $\mu$ is $K$--bi--invariant, then $c_{\pi} = 0$.
\end{prop}

\begin{proof} By Proposition \ref{FT1}(3),
$$ \widehat{\mu}(\pi) = c_{\pi}I_{\pi} = E^{K}_{\pi}\widehat{\mu}(\pi)E^{K}_{\pi},$$
from which we deduce that $c_{\pi} I_{\pi} = c_{\pi}E^{K}_{\pi}$. Assume $c_{\pi} \neq 0$; since the range of $E^{K}_{\pi}$ is one--dimensional, we can find a non-zero vector in $(E^{K}_{\pi})^{\bot}$ and this yields the desired contradiction. \end{proof}

We return to the study of convolution semigroups $(\mu_{t}, t \geq 0)$. We are interested in the case where $\mu_{t}$ has a continuous density $f_{t}$ for all $t > 0$. The following theorem is essentially due to Liao \cite{LiaoN}, Theorem 4.8.

\begin{theorem} \label{Lt} Let $(\mu_{t}, t \geq 0)$ be a convolution semigroup of probability measures on the compact group $G$. The following are equivalent:
\begin{enumerate}
\item $\mu_{t}$ has an $L^{2}$--density for all $t > 0$,
\item $\mu_{t}$ has a continuous density for all $t > 0$,
\item The series $\sum_{\pi \in \G}d_{\pi}\tr(\widehat{\mu_{t}}(\pi)\pi(g))$ converges absolutely and uniformly in $g \in G$ for all $t > 0$.
\end{enumerate}
\end{theorem}

\begin{proof} (2) implies (1) is obvious as $C(G) \subseteq L^{2}(G)$ for $G$ compact. (3) implies (2) since if, for each $g \in G, t > 0$, we define
$$ f_{t}(g) = \sum_{\pi \in \G}d_{\pi}\tr(\widehat{\mu_{t}}(\pi)\pi(g)).$$
then $f_{t}$ is the uniform limit of a sequence of continuous functions on $G$ and so is continuous. The fact that $f_{t}$ is the Radon--Nikodym derivative of $\mu_{t}$ follows by the argument of Theorem 4.5.1 in \cite{App1}. To show that (1) implies (3), we present the argument given in the proof of \cite{LiaoN}, Theorem 4.8. First choose $r > 0$ and define $f_{r/2}$ as above. By the Plancherel theorem and (\ref{Apref}), $f_{r/2} \in L^{2}(G)$ and $||f_{r/2}||^{2} = \sum_{\pi \in \G}d_{\pi}||\widehat{\mu_{r/2}}(\pi)||_{HS}^{2} < \infty$. Then given any $\epsilon > 0$, there exists a finite set $\G_{0} \subset \G$ so that $$\sum_{\pi \in \G \setminus \G_{0}}d_{\pi}||\widehat{\mu_{r/2}}(\pi)||_{HS}^{2} < \epsilon^{2}.$$ Using the matrix inequality $|\tr(A^{*}B)| \leq ||A||_{HS}||B||_{HS}$ for $A, B \in M_{d_{\pi}}(\C)$, and the Cauchy--Schwarz inequality, we have for all for all $g \in G, t > r/2$,
\bean & &  \sum_{\pi \in \G \setminus \G_{0}}d_{\pi}|\tr(\widehat{\mu_{t}}(\pi)^{*}\pi(g))|\\ & \leq & \sum_{\pi \in \G \setminus \G_{0}}d_{\pi}|\tr(\widehat{\mu_{r/2}}(\pi)^{*}\widehat{\mu_{t-r/2}}(\pi)^{*}\pi(g))| \\
& = & \sum_{\pi \in \G \setminus \G_{0}}d_{\pi}||\widehat{\mu_{r/2}}(\pi)||_{HS}||\widehat{\mu_{t-r/2}}(\pi)||_{HS} \\
& \leq & \left(\sum_{\pi \in \G \setminus \G_{0}}d_{\pi}||\widehat{\mu_{r/2}}(\pi)||_{HS}^{2}\right)^{1/2}\left(\sum_{\pi \in \G \setminus \G_{0}}d_{\pi}||\widehat{\mu_{t-r/2}}(\pi)||_{HS}^{2}\right)^{1/2}\\
& \leq & \epsilon ||f_{t - r/2}||. \eean
\end{proof}

For the remainder of this paper, we assume that $X$ is a compact (globally Riemannian) symmetric space, so that $G$ is a compact Lie group with Lie algebra having dimension $d$.
Let $L = (L(t), t \geq 0)$ be a (left) L\'{e}vy process on $G$ so that $L$ has stationary and independent increments and is stochastically continuous (see e.g. \cite{Liao} for relevant background). For each $t \geq 0$, let $\mu_{t}$ denote the law of $L(t)$, so that $\mu_{t}(A) = P(L(t) \in A)$ for all $A \in {\mathcal B}(G)$, then $(\mu_{t}, t \geq 0)$ is a convolution semigroup of probability measures on $G$. We say that the process $L$ is $K$--bi--invariant, if $\mu_{t}$ is $K$--bi--invariant for all $t \geq 0$ (and so $\mu_{0} = m_{K}$). For $K$--bi--invariant $L$, define $Y = (Y(t), t \geq 0)$ by $Y(t) = \xi(L(t))$ for all $t \geq 0$. Then as is shown in \cite{LiaoN} (see also Theorem 3.2 in \cite{Berg1}), $(Y(t), t \geq 0)$ is a $G$-invariant Feller process on $X = G/K$, with $Y(0) = o$ (a.s.)\footnote{The most general $G$--invariant Feller process in $X$ is obtained by taking $L$ to be a $K$--conjugate--invariant L\'{e}vy process, as shown in Theorems 1.17 and 3.10 of \cite{LiaoN}; see also Theorem 2.2 in \cite{Liao}. This larger class of processes is not so convenient for the spectral theoretic considerations discussed in section 5.}   The $G$--invariance is manifest as follows: for each $t \geq 0, x \in X, A \in {\mathcal B}(X)$, let $q_{t}(x, A) = P(Y(t) \in A| Y(0) = x)$ be the usual transition probability, then for all $g \in G$:
$$ q_{t}(\tau(g)x, \tau(g)A) = q_{t}(x, A).$$ If $(Q_{t}, t \geq 0)$ is the transition semigroup of the process $Y$, then for all $t \geq 0, f \in C(X), x \in X$,
$$ Q_{t}f(x) = \int_{X}f(y)q_{t}(x, dy),$$
and as is easily verified (see also  Proposition 1.16 of \cite{LiaoN})
\begin{equation} \label{semi1}
Q_{t}f \circ \xi = P_{t}(f \circ \xi).
\end{equation}

\begin{theorem} \label{dequiv} Let $(L(t), t \geq 0)$ be a $K$--bi--invariant L\'{e}vy process on $G$, and $(Y(t), t \geq 0)$ be the projected Feller process on $X$.
If for all $t > 0, L(t)$ has a continuous density $\rho_{t}$, then $Y(t)$ has a continuous transition density $k_{t}(\cdot, \cdot)$, and for all $g, h \in G$ we have
\begin{equation} \label{dequiv1}
k_{t}(gK, hK) = \rho_{t}(g^{-1}h).
\end{equation}
\end{theorem}

\begin{proof} Using (\ref{semi1}) and (\ref{int1}), for all $t > 0, f \in C(X), g \in G$
\bean \int_{X}f(x)q_{t}(gK, dx) & = & Q_{t}f(gK)\\
& = & P_{t}(f \circ \xi)(g)\\
& = & \int_{G}(f \circ \xi)(h)\rho_{t}(g^{-1}h)dh \\
& = & \int_{X}f(x)\widetilde{\rho_{t}}(\tau(g^{-1})x)\sigma(dx),\eean
where $\widetilde{\rho_{t}}$ is the unique function in $C(X)$ so that $\rho_{t} = \widetilde{\rho_{t}} \circ \xi$. So the required transition density exists and  for all $g, h \in G$, we have by (\ref{use}),
\bean k_{t}(gK, hK) & = & \widetilde{\rho_{t}}(\tau(g^{-1})\xi(h))\\
& = & (\widetilde{\rho_{t}} \circ \xi)(l_{g^{-1}}h)\\
& = & \rho_{t}(g^{-1}h). \eean \end{proof}

\section{Eigenfunction Expansions for the Transition Kernel}

Let the processes $L$ and $Y$ be as in the previous section, so that $(\mu_{t}, t \geq 0)$ is a $K$--bi--invariant convolution semigroup on $G$. We continue to assume that $\mu_{t}$ has a continuous density $\rho_{t}$ for all $t > 0$.  We equip $G$ with an Ad--invariant Riemannian metric which is compatible with the Riemannian structure on $X$, and let $\Delta$ be the associated Laplace--Beltrami operator on $G$. Then $\{\kappa_{\pi}, \pi \in \G\}$ will denote the Casimir spectrum for $G$ so that $\kappa_{\pi} \geq 0$ (with $\kappa_{\pi} = 0$ if and only if $\pi$ is trivial) and $\Delta \phi_{\pi} = - \kappa_{\pi}\phi_{\pi}$ for all $\pi \in \G_{s}$. Assume that the symmetric space $X$ is irreducible, in that the action of Ad$(K)$ on $\p$ is irreducible, where $\p:=\g \ominus \fk$, and $\fk$ is the Lie algebra of $K$. A sufficient condition for this to hold is that $G$ is semisimple (see Proposition 5.12 in \cite{LiaoN}).

Then Gangolli's L\'{e}vy Khinchine formula (see e.g. \cite{Gang1}, \cite{Appb}, \cite{LW}) tells us that for all $t \geq 0, \pi \in \G_{s}$

\begin{equation} \label{Gang1}
\widehat{\mu_{t}}(\phi_{\pi}) = e^{- t\chi_{\pi}},
\end{equation}

where

\begin{equation} \label{Gang2}
\chi_{\pi} = a \kappa_{\pi} + \int_{G}(1 - \phi_{\pi}(g)) \nu(dg),
\end{equation}

with $a \geq 0$ and $\nu$ a $K$--bi--invariant L\'{e}vy measure on $(G, {\mathcal B}(G))$. It follows easily from Proposition \ref{nocent1} that if $G$ is non--abelian, then $\mu_{t}$ cannot be central for $t > 0$.

\begin{theorem} \label{fspecth}
Suppose that $(\mu_{t}, t \geq 0)$ is a $K$--bi--invariant convolution semigroup.
\begin{enumerate} \item For all $t \geq 0, \pi \in \G, 1 \leq i, j \leq d_{\pi}$,
\begin{equation} \label{fspec}
P_{t}\pi_{ij}  =  \left\{\begin{array}{c c} e^{-t \overline{\chi_{\pi}}}
\pi_{ij} & \mbox{if}~i = 1, \pi \in \G_{s}\\
  0 & \mbox{otherwise} \end{array} \right. \end{equation}

 \item If $\mu_{t}$ has a continuous density for all $t \geq 0$, then $P_{t}$ is trace-class in $L^{2}(G)$, and its trace is given by
 $$\mbox{Tr}(P_{t}) = \sum_{\pi \in \G_{s}}d_{\pi}e^{-t \chi_{\pi}}.$$
 \end{enumerate}
 \end{theorem}

 \begin{proof}
 \begin{enumerate}
 \item We argue as in the proof of Theorem 5.3 in \cite{App3}. First observe that since for each $1 \leq i,j \leq d_{\pi}, g,h \in G, \pi_{ij}(g): = \la \pi(g)e_{i}^{\pi}, e_{j}^{\pi} \ra$ and
 $\pi(gh) = \pi(g)\pi(h)$, then $\pi_{ij}(gh) = \sum_{k=1}^{d_{\pi}}\pi_{ik}(h)\pi_{kj}(g)$. Hence
 \bean P_{t}\pi_{ij}(g) & = & \int_{G}\pi_{ij}(gh)\mu_{t}(dh)\\
 & = & \sum_{k=1}^{d_{\pi}}\pi_{kj}(g)\int_{G}\pi_{ik}(h)\mu_{t}(dh)\\
 & = & \sum_{k=1}^{d_{\pi}}\pi_{kj}(g)\int_{G}\overline{\pi_{ki}(h^{-1})}\mu_{t}(dh)\\
 & = & \sum_{k=1}^{d_{\pi}}\pi_{kj}(g)\overline{\widehat{\mu_{t}}(\pi)_{ki}}, \eean
 and the result follows by Proposition \ref{FT1} (3) and (\ref{Gang1}).

 \item If $\mu_{t}$ has a continuous density, $P_{t}$ is trace--class by Theorem 5.4.4 in \cite{App1}. From (1), we have
  $$\mbox{Tr}(P_{t}) = \sum_{\pi \in \G_{s}}d_{\pi}e^{-t \overline{\chi_{\pi}}},$$
 but for each $\pi \in \G_{s}$, we have $\widetilde{\pi} \in \G_{s}$, where $\widetilde{\pi}$ is the conjugate representation, and the result follows when we observe that
 $\overline{\chi_{\pi}} = \chi_{\widetilde{\pi}}$.
 \end{enumerate}
 \end{proof}

In the last theorem, we calculated the spectrum of $P_{t}$ in the space $L^{2}(G)$. In the next result, we restrict to the closed subspace $L^{2}(K \backslash G/K)$.

\begin{theorem} \label{kerexp}
Suppose that $(\mu_{t}, t \geq 0)$ is a $K$--bi--invariant convolution semigroup.
\begin{enumerate}
\item For all $t \geq 0, \pi \in \G_{s}$,
$$ P_{t}\phi_{\pi} = e^{-t \chi_{\pi}}\phi_{\pi},$$
\item If $\mu_{t}$ has a continuous density for all $t \geq 0$, then for all $g,h \in G$,
\begin{enumerate}
\item
$$k_{t}(gK, hK) =  \sum_{\pi \in \G_{s}}d_{\pi}e^{-t \overline{\chi_{\pi}}}\phi_{\pi}(g^{-1}h),$$

\item %If for all $t > 0$, $\sum_{\pi\in \G_{s}}d_{\pi}e^{-t \Re(\chi_{\pi})} < \infty$, then for all $g, h \in G$,
$$k_{t}(gK, hK) = \sum_{\pi \in \G_{s}}\sum_{j=1}^{d_\pi}d_{\pi}e^{-t \chi_{\pi}}\overline{\pi_{1j}(g)}\pi_{1j}(h),$$

\end{enumerate}
\end{enumerate}
\end{theorem}

\begin{proof}
\begin{enumerate}

\item This can be deduced from Theorem \ref{fspecth}(1), but alternatively, using Fubini's theorem and (\ref{Gang1}), we have for all $g \in G$,
\bean P_{t}\phi_{\pi}(g) & = & \int_{G}\phi_{\pi}(gh)\mu_{t}(dh)\\
& = & \int_{G}\int_{K}\phi_{\pi}(gkh)dk \mu_{t}(dh)\\
& = & \widehat{\mu_{t}}(\phi_{\pi})\phi_{\pi}(g)\\
& = & e^{-t \chi_{\pi}}\phi_{\pi}. \eean

\item
\begin{enumerate}

\item By Fourier expansion in $L^{2}(K \backslash G/K)$,
\bean \rho_{t} & = &  \sum_{\pi \in \G_{s}} d_{\pi} \la \rho_{t}, \phi_{\pi} \ra \phi_{\pi}\\
& = & \sum_{\pi \in \G_{s}}d_{\pi} \overline{\widehat{\mu_{t}}(\phi_{\pi})} \phi_{\pi}, \eean
and so
$$ \rho_{t}(g^{-1}h) = \sum_{\pi \in \G_{s}}d_{\pi} e^{-t \overline{\chi_{\pi}}}\phi_{\pi}(g^{-1}h).$$

The result then follows from Theorem \ref{dequiv}.

\item As $L_{g}\rho_{t}$ is $K$--right--invariant for all $g \in G$, we may use Proposition \ref{PW}(2) to write,
$$L_{g}\rho_{t}  =   \sum_{\pi \in \G_{s}} d_{\pi} \sum_{j = 1}^{d_{\pi}} \la L_{g}\rho_{t}, \pi_{1j} \ra \pi_{1j},$$
but for each $j = 1, \ldots, d_{\pi}$, \bean \la L_{g}\rho_{t}, \pi_{1j} \ra & = & \int_{G}\rho_{t}(g^{-1}g^{\prime})\overline{\pi_{1j}(g^{\prime})}dg^{\prime} \\
& = & \int_{G}\rho_{t}(g^{\prime})\overline{\pi_{1j}(gg^{\prime})}dg^{\prime}\\
& = & \sum_{k=1}^{d_{\pi}}\overline{\pi_{kj}(g)}\int_{G}\rho_{t}(g^{\prime})\overline{\pi_{1k}(g^\prime)}dg^{\prime}\\
& = & e^{-t \chi_{\pi}}\overline{\pi_{1j}(g)}, \eean
since by $K$--bi--invariance of $\rho_{t}, \la \rho_{t}, \pi_{1j} \ra = 0$ for all $j \neq 1$, and the result follows easily from here.

\end{enumerate}
\end{enumerate}
\end{proof}

It is interesting to compare Theorem \ref{kerexp} (2) (a) with results obtained by Bochner for spheres (see \cite{Bo} p.1146). In the case of the heat kernel (so $\nu = 0$ in (\ref{Gang2})), a formula of this type on general compact homogeneous spaces is presented in \cite{Ben}.

%In Theorem \ref{kerexp}, we calculated the spectrum of $P_{t}$ in the space $L^{2}(K \backslash G /K)$. We also find it useful to do this in the space $L^{2}(G)$, and using Proposition \ref{FT1}(3) and the argument of Theorem 5.3 in \cite{App3}, it is not difficult to verify that for all $t \geq 0, \pi \in \G, 1 \leq i, j \leq d_{\pi}$,

%\begin{equation} \label{fspec}
%P_{t}\pi_{ij}  =  \left\{\begin{array}{c c} e^{-t \overline{\chi_{\pi}}}
%\pi_{ij} & \mbox{if}~j = 1, \pi \in \G_{s}\\
% = 0 & \mbox{otherwise} \end{array} \right. \end{equation}

%It follows from Theorem 5.4.4 in \cite{App1} that for each $t > 0$,  $P_{t}$ is trace--class in $L^{2}(G)$, and so the condition of Theorem \ref{kerexp} (3) is always satisfied in our context.

We now easily deduce the following trace formula:

\begin{cor} \label{trform} If $\mu_{t}$ is $K$--bi--invariant and has a square--integrable density for all $t > 0$, then

$$ k_{t}(x,x) = k_{t}(o, o) =  \mbox{Tr}(P_{t}).$$

\end{cor}

\begin{proof} This follows on putting $g = h = e$ in Theorem \ref{kerexp} (2). \end{proof}

{\bf Notes} \begin{enumerate}
\item It is interesting to compare the results obtained herein with those in section 5 of \cite{App3}. We did not need to assume that the  convolution semigroup is central in order to obtain a ``scalar'' L\'{e}vy-Khintchine formula. That follows from $K$--bi--invariance.

\item The formulae for the trace in the two papers are different, in that a factor of $d_{\pi}^{2}$ in the sum has reduced to $d_{\pi}$. This is because (as seen in (\ref{fspec})), the eigenspace for each eigenvalue is spanned by the top row of the representation matrix, rather than the entire set of matrix entries.

\item It is also of interest to calculate the trace Tr$_{K}(P_{t})$ of the semigroup on the space $L^{2}(K \backslash G / K)$. It follows from Theorem \ref{kerexp} (1) (see also section 3 of \cite{App4}), that for each $t > 0$,
$$ \mbox{Tr}_{K}(P_{t}) =    \sum_{\pi \in \G_{s}}e^{-t \chi_{\pi}}.$$

\end{enumerate}

A standard convolution semigroup $(\mu_{t}, t \geq 0)$ is said to be central if $\mu_{t}$ is central for all $t \geq 0$, and it is {\it symmetric} if $\mu_{t}$ is a symmetric measure, i.e. $\mu_{t} = \mu_{t}^{\prime}$ for all $t \geq 0$. Clearly if $(\mu_{t}, t \geq 0)$ is symmetric, then so is $(\mu_{t}^{(K)}, t \geq 0)$. Moreover, it follows from Theorem 2.2 in \cite{App4} (or Theorem 5.4.1 in \cite{App1}) that $P_{t}$ is self--adjoint in $L^{2}(G)$, and the L\'{e}vy measure $\nu$ appearing in (\ref{Gang1}) is symmetric. If $(\mu_{t}, t \geq 0)$ is symmetric, then $\chi_{\pi} \geq 0$ for all $\pi \in \G_{s}$,
%$$ \mbox{Tr}(P_{t}) = \sum_{\pi \in \G_{s}}d_{\pi}e^{-t \chi_{\pi}},$$
and the trace formula of Corollary \ref{trform} is a special case of Mercer's theorem (see e.g. \cite{Dav}, pp.156--7).

Well--known examples of central symmetric convolution semigroups having $C^{\infty}$ densities for $t > 0$, are the Gaussian (heat) semigroup where for all $\pi \in \G, t \geq 0$,
$c_{\pi}(t) = e^{- a t\kappa_{\pi}}$ for some $a > 0$, and the $\alpha$-stable type semigroup where $c_{\pi}(t) = e^{- a t\kappa_{\pi}^{\alpha}}$ for $0 < \alpha < 1$ (see Proposition 5.8.1 in \cite{App1}, pp.157--8). A rather wide class of examples that fit into the context of this section, are obtained by imposing $a > 0$ in (\ref{Gang1}). Then for each $t > 0, \mu_{t}$ is the convolution of a Gaussian measure (as described above) with the law of a pure jump L\'{e}vy process, and $\mu_{t}$ has a $C^{\infty}$ density by Corollary 4.5.1 in \cite{App1}, p.103 (see also Theorem 3 in \cite{LW}).  In each of the above cases, the measure $\mu_{t}^{(K)}$ also has a $C^{\infty}$ density for $t > 0$ by Theorem \ref{inher}.

If $\psi = (\psi(t), t \geq 0)$ is Brownian motion on $G$, then its laws $(\mu_{t}, t \geq 0)$ give the flow of heat kernel measures, and these are central and symmetric, as discussed. In this case $(\mu_{t}^{(K)}, t \geq 0)$ are the laws of $K$--bi--invariant (spherical) Brownian motion on $G$, $\widetilde{\psi} = (\widetilde{\psi}(t), t \geq 0)$. For $t  > 0$, these measures are symmetric, but not central when $G$ is non--abelian. It is interesting to look at these processes from the point of view of stochastic differential equations (sdes). Let $\{X_{1}, \ldots, X_{d}\}$ be an orthonormal basis for $\g$ (with respect to the given Ad-invariant inner product), such that $X_{1}, \ldots, X_{m} \in \p$ and $X_{m+1}, \ldots, X_{d} \in \fk$. Let $B = (B_{1}, \ldots, B_{d})$ be a standard Brownian motion in $\Rd$. Then $\psi$ is the unique solution to the sde
$$ d\psi(t) = \sum_{i=1}^{d}X_{i}(\psi(t)) \circ dB_{i}(t)~, \psi(0) = e~\mbox{(a.s.)},$$
while $\widetilde{\psi}$ is the unique solution to
$$ d\widetilde{\psi}(t) = \sum_{i=1}^{m}X_{i}(\widetilde{\psi}(t)) \circ dB_{i}(t)~, \widetilde{\psi}(0) = U_{K}~\mbox{(a.s.)},$$
for $t \geq 0$, where $U_{K}$ is uniformly distributed on $K$ and $\circ$ denotes the Stratonovitch differential. If $\Delta_{G} = \sum_{i=1}^{d}X_{i}^{2}$ is the group Laplacian, then for $t > 0$ the heat kernel $\kappa_{t}$, which is the density of $\mu_{t}$, is the fundamental solution of the pde $\frac{\partial u(t)}{\partial t} = \Delta_{G}u(t)$. The spherical heat kernel $\rho_{t}$, which is the density of $\mu_{t}^{(K)}$, is the fundamental solution of $\frac{\partial u(t)}{\partial t} = \Delta_{\p}u(t)$, where the ``horizontal Laplacian'' $\Delta_{\p}: = \sum_{i=1}^{m}X_{i}^{2}$. For further details and discussion, see \cite{Appb} and section 3.4 of \cite{LiaoN}.

We close this section by giving a brief account of subordination and short time asymptotics. For background on subordination in compact Lie groups, we refer the reader to section 5.7 of \cite{App1}, and to \cite{AG}. Let $(S(t), t \geq 0)$ be a subordinator  having law $\lambda_{t}$ for  $t \geq 0$, that is independent of the L\'{e}vy process $(L(t), t \geq 0)$. Then for all $u > 0, \int_{[0, \infty)}e^{- u s}\lambda_{t}(ds) = e^{-t \psi(u)}$, where $\psi$ is a Bernstein function such that $\lim_{u \rightarrow 0+}\psi(u) = 0$, so that for all $u > 0$,

\begin{equation} \label{Bern}
\psi(u) = bu + \int_{(0, \infty)}(1 - e^{-yu})\tau(dy),
\end{equation}

\noindent with $b \geq 0$ and $\tau$ a L\'{e}vy measure on $(0, \infty)$, i.e. $\int_{(0, \infty)}(1 \wedge y)\tau(dy) < \infty$.

We subordinate to form a new L\'{e}vy process $L^{S}(t) = L(S(t))$, having law
$\mu^{S}_{t}(A) = \int_{0}^{\infty}\mu_{s}(A)\lambda_{t}(ds)$ for each $t \geq 0, A \in {\mathcal B}(G)$. It is clear that if $L$ is $K$--bi--invariant, then so is $L^{S}$, and we make this assumption henceforth. Then for all $\pi \in \G_{s}, t \geq 0$

\begin{equation} \label{FTsub}
\widehat{\mu^{S}_{t}}(\phi_{\pi}) = e^{-t \psi(\chi_{\pi})}.
\end{equation}

The subordinated semigroup $(P_{t}^{S}, t \geq 0)$, which is the transition semigroup of the process $L^{S}$, is defined as

$$ P_{t}^{S}f(g) = \int_{0}^{\infty}P_{s}f(g)\lambda_{t}(ds), $$

for all $f \in L^{2}(G), g \in G, t \geq 0$, and by (\ref{fspec}) and (\ref{FTsub}) we deduce that for all $t \geq 0, \pi \in \G, 1 \leq i, j \leq d_{\pi}$,

\begin{equation} \label{fspecsub}
P_{t}^{S}\pi_{ij}  =  \left\{\begin{array}{c c} e^{-t \overline{\psi(\chi_{\pi})}}
\pi_{ij} & \mbox{if}~i = 1, \pi \in \G_{s}\\
 0 & \mbox{otherwise} \end{array} \right. \end{equation}

If $L(t)$ has a density $\rho_{t}$ for all $t > 0$, then $L^{S}(t)$ has a density $\rho^{S}_{t}$ given by
$\rho^{S}_{t}(g) = \int_{0}^{t}\rho_{s}(g)\lambda_{t}(ds)$, for each $g \in G$. %We assume from now on that $\rho^{S}_{t}$ is continuous for all $t > 0$. The next result gives some conditions that enable this to hold.
From now on, let $L$ be $K$--bi--invariant Brownian motion on $G$ (denoted $\widetilde{\psi}$ above), so that $k_{t}$ is the heat kernel on $X$:
$$ k_{t}(gK,hK) = \sum_{\pi \in \G_{s}}d_{\pi}e^{-t \kappa_{\pi}}\phi_{\pi}(g^{-1}h),$$
for $t > 0, g,h \in G$. Then as $t \rightarrow 0$, we have the well-known asymptotic behaviour (see e.g. \cite{Feg}):
$$ k_{t}(o, o) \sim \frac{\mbox{Vol(X)}}{(4\pi)^{d/2}}t^{-d/2}.$$

%\href{}{}Let $\alpha = \min_{\pi \in $

\begin{theorem}  If $b > 0$  or $\gamma: = \inf_{y > 0}\int_{(0,1)}ue^{- uy}\tau(du) > 0$ then $\rho^{S}_{t}$ is continuous for all $t > 0$.
\end{theorem}

\begin{proof} We will show that, under the stated condition,  $P_{t}^{S}$ is trace--class for all $t > 0$. Then $\rho^{S}_{t}$ exists and is square--integrable by Theorem 5.4.4 in \cite{App1}. It follows that $\rho^{S}_{t}$ is continuous by Theorem \ref{Lt}. Using (\ref{fspecsub}) and (\ref{Bern}) and the fact that $\widehat{\mu_{t}}(\phi_{\pi}) = e^{-t\kappa_{\pi}}$ for all $\pi \in \G_{s}$, we find that for some $0 < \theta < 1$,
 \bean \mbox{Tr}(P_{t}^{S}) & = & \sum_{\pi \in \G_{s}}d_{\pi}e^{-t \psi(\kappa_{\pi})}\\
 & \leq &  \sum_{\pi \in \G_{s}}d_{\pi}e^{-tb \kappa_{\pi}}\exp{\left\{-t \int_{(0, 1)}(1 - e^{-u \kappa_{\pi}})\tau(du)\right\}} \\
 & = & \sum_{\pi \in \G_{s}}d_{\pi}e^{-tb \kappa_{\pi}}\exp{\left\{-t \kappa_{\pi} \int_{(0, 1)}ue^{- \theta u \kappa_{\pi}}\tau(du)\right\}} \\
 %& \leq & \sum_{\pi \in \G_{s}}d_{\pi}e^{-tb \kappa_{\pi}}\exp{\left\{-t \int_{0}^{1}u\tau(du) \kappa_{\pi}e^{-\kappa_{\pi}} \right\}} \\
 & \leq & \sum_{\pi \in \G_{s}}d_{\pi}e^{-tb \kappa_{\pi}} e^{-t \gamma \kappa_{\pi}}. \eean
  If $b > 0$, we have
 $$ \mbox{Tr}(P_{t}^{S}) \leq \sum_{\pi \in \G_{s}}d_{\pi}e^{-tb \kappa_{\pi}} < \infty,$$
 since the right hand side is the trace of a heat kernel semigroup (with variance $b$) which we know to be finite. The other case is similar.

 \end{proof}

We assume from now on that $\rho^{S}_{t}$ is continuous for all $t > 0$.

%{\bf Note.} We included the jump term in the above estimate to illustrate that it is difficult to find a straightforward condition on this part that guarantees summability. Nonetheless, we know this must be possible as it works for $\alpha$--stable subordinators.

%\vspace{5pt}

We obtain a subordinated $G$--invariant Feller process $Y^{S}$ on $X$, where for all $t \geq 0, Y^{S}(t): = \xi(L^{S}(t))$. Then $Y^{S}$ inherits a transition density from $Y$ which is given by
$$ k_{t}^{S}(x, y) = \int_{(0, \infty)}k_{s}(x, y)\lambda_{t}(ds),$$
for $t \geq 0, x, y \in X$. By Theorem \ref{dequiv}, we have $k_{t}^{S}(gK, hK) = \rho_{t}^{S}(g^{-1}h)$, and so $k_{t}^{S}$ inherits joint continuity from the assumed continuity of $\rho^{S}_{t}$.

By Theorem \ref{kerexp}, we have the Fourier expansion:

$$k_{t}^{S}(gK, hK) = \sum_{\pi \in \G_{s}}d_{\pi}e^{-t \psi(\kappa_{\pi})}\phi_{\pi}(g^{-1}h),$$
for all $g, h \in G, t > 0$.

%Bertoin p.75

If we assume that the Bernstein function $\psi$ has an increasing inverse $\psi^{-1}$, and that $\psi$ is regularly varying at infinity with index $r > 0$, then the result of \cite{BaBa} yields, as $t \rightarrow 0$,

\begin{equation} \label{asym}
k_{t}^{S}(o, o) \sim \frac{\mbox{Vol(X)}\Gamma(d/2r + 1)}{(4\pi)^{d/2}\Gamma(d/2 + 1)}\psi^{-1}(1/t)^{d/2}.
\end{equation}

In particular, if we take $\psi(t) = t^{\alpha}$, so that $\widehat{\mu^{S}_{t}}(\phi_{\pi}) = e^{- \kappa_{\pi}^{\alpha}}$ for $0 < \alpha < 1$, then
$$ k_{t}^{S}(o, o) \sim \frac{\mbox{Vol(X)}\Gamma(d/2r + 1)}{(4\pi)^{d/2}\Gamma(d/2 + 1)}t^{-d/2\alpha}.$$

For a sub--class of subordinators, which includes many important examples, a more explicit asymptotic (series) expansion, which generalises that of the heat kernel, can be found in \cite{Fa}.

\section{Invariant Feller Processes on the Sphere}

Let $S:= S^{d-1}$ be the $(d-1)$--dimensional unit sphere embedded in $\R^{d}$ (where $d \geq 3$), so that
$$ S^{d}: = \{x = (x_{1}, \ldots, x_{d}) \in \R^{d}; ||x|| = 1\}.$$
Then $G$ is a compact symmetric space with $G = SO(d)$ and $K = SO(d-1)$. As is conventional, we take the point $o$ to be the ``north pole'' $e_{d}$, where $(e_{1}, \ldots, e_{d})$ is the natural basis in $\Rd$. The introductory material that follows is mainly based on \cite{Far}, Chapter 9. The required Ad--invariant metric on $G$ is obtained by equipping its Lie algebra ${\bf so(d)}$ with the negation of its Killing form to induce the inner product
$$ \la A, B \ra = (2-d)\tr(AB),$$
for each $A, B \in {\bf so(d)}$.

The double cosets are the orbits of $K$ in $S$, and these are themselves spheres of dimension $d-2$. These ``parallels'' may be labelled by the co-latitude $\theta \in [0, \pi]$. To make this more precise, observe that the mapping $\zeta: (0, \pi) \times S^{d-2} \rightarrow S^{d-1} \setminus \{\pm e_{d}\}$ is a diffeomorphism, where for each $\theta \in (0, \pi), y \in S^{d-2}$,
$$ \zeta((\theta, y)) = \sin(\theta)y + \cos(\theta)e_{d}.$$
From this we deduce that a continuous mapping $f:S \rightarrow \R$ is $K$--bi--invariant, if and only if it is zonal, i.e. for all $x \in S \setminus \{\pm e_{d}\}$,
$ f(x) = (f \circ \zeta^{-1})(\theta, u)$ depends only on $\theta$, and so we may write $f(x) = F(x_{d})$ for all $x \in S$,
where $F: [-1, 1] \rightarrow \R$ is continuous. For such zonal functions, we have the integral formula:

\begin{equation} \label{Sintf}
\int_{S}f(x)\sigma(dx) = \frac{\Gamma(\frac{d}{2})}{\sqrt{\pi}\Gamma(\frac{d-1}{2})}\int_{-1}^{1}F(t)(1 - t^{2})^{\frac{d-3}{2}}dt,
\end{equation}
so in (\ref{doubleinv}), we have $E = [-1, 1]$ and $\alpha_{d}(dt) = \frac{\Gamma(\frac{d}{2})}{\sqrt{\pi}\Gamma(\frac{d-1}{2})}(1 - t^{2})^{\frac{d-3}{2}}dt$.

The irreducible representation of $SO(d)$ are all spherical, and are indexed by $\Z_{+}$. They act on the spaces ${\mathcal H}_{n}$ of spherical harmonics that have dimension $d_{n}$ for $n \in \Z_{+}$, where:
$$ d_{n} = \binom{d + n -1}{d-1} - \binom{d + n -3}{d-1}.$$
For all $\nN$, there is a unique spherical function $\phi_{n}$ in ${\mathcal H}_{n}$ which is normalised and $K$--invariant. These functions are given in terms of $p_{n}^{d}:[-1, 1] \rightarrow \R$ by
$$ \phi_{n}(\theta) = p_{n}^{d}(\cos(\theta)) = \int_{-1}^{1}(\cos(\theta) + iy \sin(\theta))^{n}\alpha_{d-1}(dy).$$
If $d = 3$, then $p_{n}^{d}$ is a Legendre polynomial. More generally, for $d \geq 3$, the $p_{n}^{d}$'s are related to the ultraspherical (Gegenbauer) polynomials $G_{n}^{\nu}$ (where $\nu \in [0, \infty)$) as follows:
$$ p_{n}^{d}(t) = \binom{n+d-3}{n}^{-1}G_{n}^{\frac{d-2}{2}}(t),$$
for all $t \in [-1, 1]$. Finally we have the generating function identity:
$$ \sum_{n=0}^{\infty}r^{n}G_{n}^{\nu}(t) = \frac{1}{(1 + r^{2} - 2rt)^{\nu}},$$
for $r \in [0, 1)$ (see \cite{Mull}, pp. 44--50 for details).

The Laplace--Beltrami operator diagonalises on ${\mathcal H}_{n}$, and for each $n \in \Z_{+}$, we have
$$ \Delta \phi_{n} = -n(n+d-2)\phi_{n},$$
so the Casimir spectrum is given by $\kappa_{n} = n(n+d -2)$. The Gangolli L\'{e}vy--Khintchine formula (\ref{Gang2}) then takes the form

\begin{equation} \label{Gangs}
\chi_{n} = a n(n+d -2) + \int_{0}^{\pi}(1 - p_{n}(\cos(\theta))\nu(d\theta),
\end{equation}
where $a \geq 0$ and $\int_{0}^{\pi}(1 - \cos(\theta))\nu(d\theta) < \infty$ (see Theorem 3 in \cite{Bo}, \cite{Hey1}, and \cite{Gal}, Theorem 5.1 for the case $d=3$ within a more general context).

Now let $Y$ be a $SO(d)$--invariant Markov process on $S$ having a continuous transition density. Hence for each $t > 0, f \in L^{2}(S), x \in S$, we may write the transition semigroup
$$ (Q_{t}f)(x) = \int_{S}k_{t}(x,y)f(y)\sigma(dy).$$
Then as shown in \cite{Far} pp.204--5 for more general integral operators having this form, the $SO(d)$--invariance of the kernel determines the existence of a continuous (non--negative) real-valued function $a_{t}$ on $[-1,1]$ so that
$$ k_{t}(x,y) = a_{t}(x \cdot y),$$
for all $t > 0, x, y \in S$, where $\cdot$ denotes the usual scalar product in $\R^{d}$. The Chapman--Kolmogorov equations take the form
$$ a_{s+t}(x \cdot y) = \int_{S}a_{s}(x  \cdot z)a_{t}( z \cdot y)\sigma(dz),$$ for each $s, t \geq 0$.

From the results of the previous section, we know that for each $t > 0, n \in \Z_{+}, {\mathcal H}_{n}$ is an eigenspace for the operator $Q_{t}$, and that the eigenvalue $e^{-t \chi_{n}}$ has multiplicity $d_{n}$. But by the Funk--Hecke theorem (see Theorem 9.5.3 in \cite{Far} p.205--6), we have for all $t > 0, n \in \Z_{+}$,

$$ e^{-t \chi_{n}} = \frac{\Gamma(\frac{d}{2})}{\sqrt{\pi}\Gamma(\frac{d-1}{2})}\int_{-1}^{1}a_{t}(s)p_{n}(s)(1 - s^{2})^{\frac{d-3}{2}}ds.$$
%Then if $t \rightarrow a_{t}(s)$ is differentiable at $t = 0$ for all $s \in [-1, 1]$ and $s \rightarrow a^{\prime}_{0}(s)$ is continuous on $[-1,1]$, we obtain the intriguing formula
%$$ \chi_{n} = -\frac{\Gamma(\frac{d}{2})}{\sqrt{\pi}\Gamma(\frac{d-1}{2})}\int_{-1}^{1}a_{0}^{\prime}(s)p_{n}(s)(1 - s^{2})^{\frac{d-3}{2}}ds.$$
It would be interesting to determine the class of all such functions $a: (0,\infty) \times [-1, 1] \rightarrow [0, \infty)$ that arise in this way.

\vspace{5pt}

\begin{center}
{\it Acknowledgement}
\end{center}

We thank Ming Liao for making \cite{LiaoN} available to us at an early stage, for helpful conversations, and invaluable comments on an early draft of the paper.

\end{document}